\documentclass[10pt,a4paper,leqno,oneside]{article}
\usepackage[T1]{fontenc}       %
\usepackage[english]{babel}     %

\usepackage[pdftex]{graphicx}          % Pakiet pozwalaj±cy ,,wklejać'' grafikę...

\usepackage{amsmath}
\usepackage{amssymb}
\usepackage{amsthm}
\usepackage{anysize}

\theoremstyle{definition}
\newtheorem*{pf}{Proof}
\renewenvironment{proof}[1]{\begin{pf}[#1]}{\hfill\rule[0.025cm]{0.21cm}{0.21cm}\end{pf}}

\theoremstyle{plain}
\newtheorem{thm}{Theorem}
\newtheorem{lem}[thm]{Lemma}
\newtheorem{claim}[thm]{Claim}
\newtheorem{prop}[thm]{Proposition}

\newtheorem{cor}[thm]{Corollary}

\theoremstyle{definition}
\newtheorem{df}[thm]{Definition}

\theoremstyle{remark}
\newtheorem{rem}[thm]{Remark}

\renewcommand{\Pr}{\mathrm {P}}

\newcommand{\HH}{{\mathbb{H}^2}}
\newcommand{\HHH}{{\mathbb{H}^3}}

\newcommand{\op}{\omega^{(p)}}
\newcommand{\opd}{\omega^{(p)\dag}}
\newcommand{\g}{\gamma}
\renewcommand{\d}{\delta}

\newcommand{\G}{\Gamma}
\newcommand{\GF}{\Gamma_{\mathrm{F}}}

\newcommand{\bd}{\partial\,}
\newcommand{\cX}{{\hat{X}}}
\newcommand{\cl}[1]{{\overline{#1}}}
\renewcommand{\c}{{\mathrm c}}
\newcommand{\clX}[1]{{\overline{#1}^X}}
\newcommand{\clcX}[1]{{\overline{#1}^\cX}}
\newcommand{\cH}{{\widehat{\mathbb{H}}^2}}

\newcommand{\clcH}[1]{{\overline{#1}^\cH}}
\newcommand{\sm}{\setminus}

\newcommand{\Isom}{\mathrm {Isom}}
\newcommand{\pc}{{p_\mathrm c}}
\newcommand{\pu}{{p_\mathrm u}}
\newcommand{\p}{{p_{1/2}}}
\newcommand{\imic}{infinitely many infinite clusters}

\newcommand{\ics}{infinite clusters}
\newcommand{\ic}{infinite cluster}
\newcommand{\hyp}{hyperbolic}
\newcommand{\as}{a.~s.}
\newcommand{\vttg}{vertex-transitive tiling graph}
\newcommand{\vttgs}{vertex-transitive tiling graphs}
\newcommand{\qi}{quasi-isometry}
\newcommand{\qic}{quasi-isometric}

\renewcommand{\{}{\left\lbrace}
\renewcommand{\}}{\right\rbrace}

\newlength{\cupdotwidth}
\newcommand{\cupdotsymb}{\makebox[\cupdotwidth]
  {\makebox[0pt]{$\displaystyle{\cup}$}\makebox[0pt]{$\cdot$}}}
\DeclareMathOperator*{\cupdot}{\cupdotsymb}

\newlength{\bigcapKXwidth}
\newcommand{\bigcapKX}{\makebox[\bigcapKXwidth]{$\displaystyle{\bigcap_{\substack{K\subseteq X\\ K\text{ -- compact}}}}$}}

\begin{document}

\title{Clusters in middle-phase percolation on~hyperbolic~plane}%\footnote{To appear in Banach Center Publications in Proceedings of the ``13th WORKSHOP: NON-COMMUTATIVE HARMONIC ANALYSIS''}}

%\abbrevauthors{J.~Czajkowski} 
%\abbrevtitle{Clusters in middle-phase percolation on~hyperbolic~plane}

\author{Jan Czajkowski\\
\small Mathematical Institute, University of Wroc{\l}aw\\
\small pl. Grunwaldzki 2/4, 50-384 Wroc{\l}aw, Poland\\
\small E-mail: {\tt czajkow@math.uni.wroc.pl}}

%\maketitlebcp

\maketitle

\begin{abstract}
I consider $p$-Bernoulli bond percolation %(I will write ``$p$-percolation'' or ``percolation'' for short)
on graphs of vertex-transitive tilings of the \hyp{} plane with finite sided faces (or, equivalently, on transitive, nonamenable, planar graphs with one end) and on their duals. It is known from \cite{BS} that in such a graph $G$ we have three essential phases of percolation, i.~e.
$$0<\pc(G)<\pu(G)<1,$$
where $\pc$ is the critical probability and $\pu$ -- the unification probability.
I prove that in the middle phase a.~s.~all the ends of all the \ics{} have one-point boundary in $\bd\HH$. This result is similar to some results in \cite{Lal}.
\end{abstract}

\section{Introduction}
%\section{Preliminaries}

For any graph $G$, let $V(G)$ denote its set of vertices and $E(G)$ -- its set of edges. A \textbf{percolation} on $G$ is a random subgraph of $G$ or, one can say, a probability measure on the space of subgraphs of $G$.
For any infinite connected graph $G$ and $p\in[0;1]$ let $\omega^{(p)}(G)$ denote the process of \textbf{$p$-Bernoulli bond percolation} on $G$, which is a random subgraph of $G$ formed by taking stochastically independently each edge of $G$ with probability $p$ to the random graph and taking all the vertices of $G$ to it. The components of $\op(G)$ are often called \textbf{clusters}. One can say that in some sense the clusters of $\op$ increase with the value of parameter $p$.\footnote{That is formalized in \cite{grim}, chapter 2.1.} When $p$ increases from $0$ to $1$, first we have \as{} no \ics{} and, suddenly, above some treshold we have \as{} some \ic{} in $\op$. When we let $p$ increase further above that treshold, it turns out that in the case of \vttg{} in the hyperbolic plane $\HH$ we have \as{} \imic{} in $\op$ for some period of time, and then, after another treshold, we get exactly one \ic{} till the value of $1$ (those infinitely many clusters ,,merge'' into one). Therefore we say about three phases of Bernoulli bond percolation in such graphs.\footnote{See also remark \ref{pureph}.} Let us define precisely those tresholds.
The \textbf{critical probability} (or \textbf{critical parameter}) $\pc(G)$ of any graph $G$ is defined to be the infimum of $p\in[0;1]$ such that \as{} there is some \ic{} in $\omega^{(p)}(G)$. Similarly, the \textbf{unification probability} $\pu(G)$ is the infimum of $p\in[0;1]$ such that $\omega^{(p)}(G)$ has exactly one \ic{} \as{}

\begin{figure}[h]
  \center
  \includegraphics[width=\textwidth,bb=0 0 563 182]{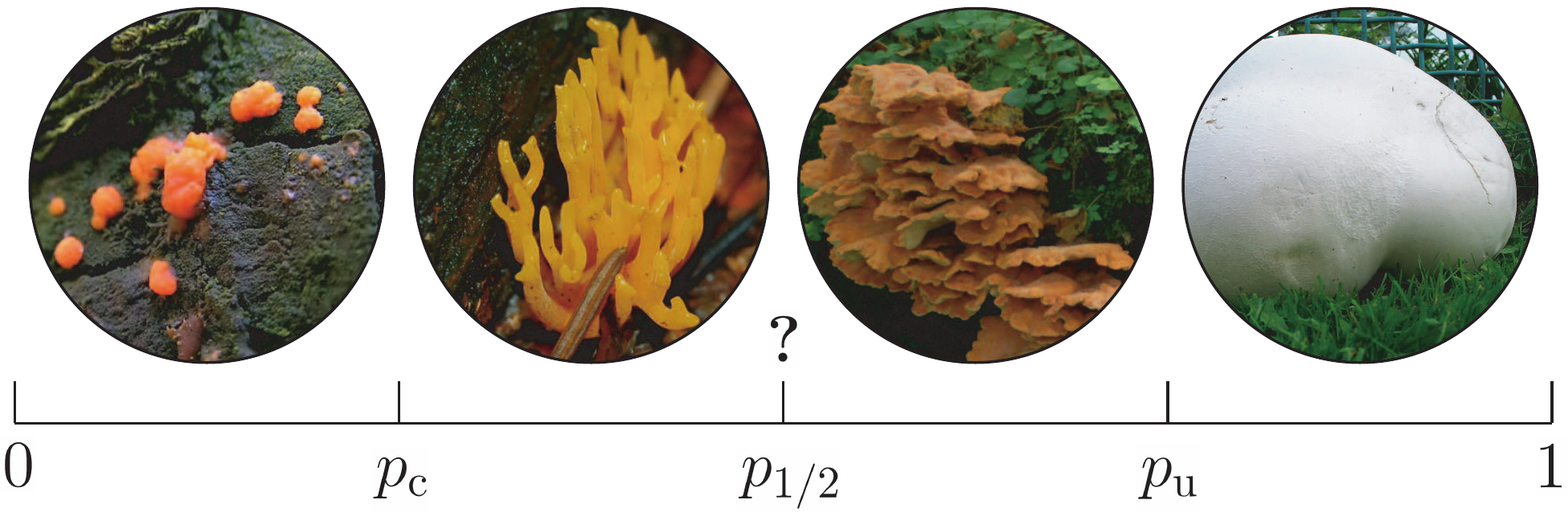}
  \caption{The idea of four phases of percolation in $\HHH$. Authors of the photos (from the left): 1.: Dorota R. ({\tt radzido}); 2., 3.: Agata Piestrzy\'nska-Kajtoch; 4.: Kazimiera Stelmach.\label{4ph}}
\end{figure}

\par A couple of important book on percolation, including the basics of percolation, are \cite{grim} and \cite{pl}. There is also paper \cite{BS} considering Bernoulli percolation on planar graphs in $\HH$ and having references on percolation on other planar graphs (e.~g. trees and lattice $\mathbb{Z}^2$); I base on that paper in this work.

\par The motivation for investigating the boundaries of ends of \ics{} comes from considering percolation phases in case of regular tilings of $\HH$ and the $3$-dimensional hyperbolic space $\HHH$.
Let us visualize $\HH$ and $\HHH$ as the Poincar\'e disc models.
\par On graphs of regular tilings of $\HHH$ we conjecturally have three phases of percolation. (It is due to conjecture 6 and question 3 in \cite{BS96}, see also theorem 10 of \cite{BB}. Inequality $\pc(G)<\pu(G)$ has been also established for some Cayley graphs of all nonamenable groups in \cite{PSN} and in some kind of continuous percolation in $\mathbb{H}^n$ in \cite{Tyk}.) So in the first phase (for $0\le p\le\pc$) we have a.~s.~only finite clusters, which roughly look like points (in large scale), so $0$-dimensional objects. In the last phase (for $\pu\le p\le 1$) there is only one big one-ended infinite cluster (one-ended means: after throwing out a bounded set it still has only one infinite component), so it looks like the whole Poincar\'e ball, which is of dimension 3. The conjecture of my advisor is that in the middle phase we have a.~s.~``$1$-dimensional'' (fibrous) \ics{} with $p$ below some treshold $\p\in(\pc;\pu)$ and ``$2$-dimensional'' (fan-shaped) with $p$ above $\p$ (see fig.~\ref{4ph}). So we should have $4$ phases -- one more than the dimension of the space ($\HHH$).
\par Following this idea, in the percolation on a graph of tiling of $\HH$ we should have only three such phases. We already know three phases of it by \cite{BS} - see theorem \ref{1.1}, so we want the clusters to be $0$-dimensional in the first phase, $1$-dimesional in the second and $2$-dimesional in the third. 
\par My formalization of $1$-dimesional is the following: all the ends of the infinite cluster have one-point boundary (which is to be explained further). The main result in this paper says that in the middle phase (for a graph of vertex-transitive tiling of $\HH$) the \ics{} are a.~s.~$1$-dimesional.

\subsection{Acknowledgement}
I wish to express gratitude to my advisor, Jan Dymara, who once proposed me percolation as master thesis topic and led me through doing it. My master thesis has developed to this article.

\section{Boundaries of ends}

Now I'm going to define the boundary of an end of an infinite cluster in $\HH$, but the definition is formulated in general.

\begin{df} \label{dfbd}
Let $X$ be a completely regular Hausdorff ($\mathrm T_{3\frac{1}{2}}$), locally compact topological space. Then:
\begin{itemize}
\item An end of a subset $a\subseteq X$ is a function $e$ from the family of all compact subsets of $X$ to the family of subsets of $a$ such that:%satisfying the following conditions:
\begin{itemize}
 \item for any compact $K\subseteq X$ the set $e(K)$ is one of the component of $a\setminus K$;
 \item for $K\subseteq K'\subseteq X$ -- both compact -- we have
 $$e(K)\supseteq e(K').$$
\end{itemize}
\end{itemize}
Now let $\cX$ be an arbitrary compactification of $X$. Then
\begin{itemize}
\item The boundary of $a\subseteq X$ is the following:
$$\bd a= \clcX{a}\setminus X$$
(by $\overline{a}^Y$ I mean the closure taken in the space $Y$).
\item Finally the boundary of an end $e$ of $a\subseteq X$ is%(considered as a subset of $\cX$) is
$$\bd e=\bigcapKX\bd e(K).$$
\end{itemize}
\end{df}

\begin{figure}[h]
  \center
  \includegraphics[width=0.5\textwidth,bb=0 0 365 435]{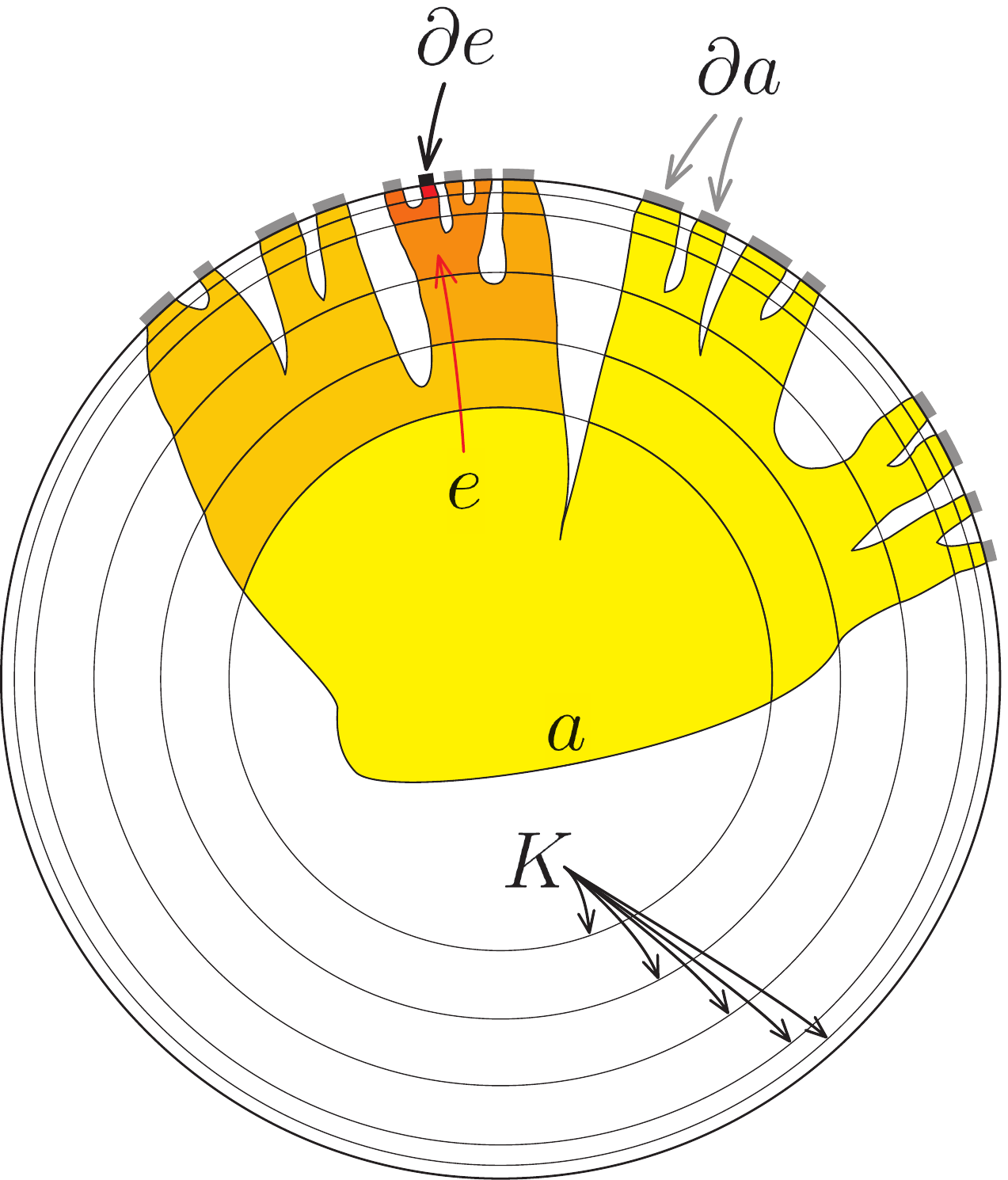}
  \caption{An end $e$ of a set $a$, its boundary $\bd e$ and the boundary $\bd a$ of the whole set in case of Poincar\'e disc.\label{bd}}
\end{figure}

In this paper I always take $X=\HH$ and $\cX=\cH$ (the closed Poincar\'e ball, i.~e. $\cH=\HH\cup\bd\HH$). The role of $a$ will be played by clusters of percolation in $\HH$.

\section{The graph}

Now I introduce some notions needed to explain what class of graphs I am considering.

\begin{df} %\label{park}
A \textbf{polygonal tiling} of $\HH$, or \textbf{tiling} of $\HH$ for short, is a family of hyperbolic polygons (in this paper by a \textbf{polygon} I mean only a finite sided polygon) which covers the \hyp{} plane, in such way that they have pairwise disjoint interiors and any two different of them are either disjoint, or intersect exactly at a sum of some of their sides and vertices.
The \textbf{graph of such tiling} as above is just the graph obtained from all the vertices and edges of the tiling. Obviously such graph is always a planar graph.
A \textbf{regular tiling} is a polygonal tiling by congruent regular polygons (regular polygons means: equilateral and equiangular).
\par A \textbf{plane graph} is a geometric realization of a planar graph in the plane (here in the definitions only the topology plays a role, so it does not matter if the plane is hyperbolic). \textbf{Faces} of a plane graph are the components of its complement in the plane.
Here I overload the notation, calling both the abstract planar graph and its plane realization by $G$ (although it does matter, see definition \ref{dfdual} of dual graph).

\end{df}

\begin{rem}
I declare all the graphs mentioned in this paper to be \textbf{simple}, i. e. not having multiple edges or loops, and \textbf{locally finite}, i.~e. having every vertex of finite degree).
\end{rem}

Further in this paper I consider graphs of polygonal tilings of $\HH$ which are vertex-transitive in the sense that some groups of isometries of $\HH$ preserving them act on their vertices transitively. I call such graphs \textbf{\vttgs}. I consider also their duals as well.

\begin{rem}
The main theorem is proven for all \vttgs{} (theorem \ref{main}) and their duals (corollary \ref{dualmain}). Earlier in my master thesis I considered only graphs of two special regular tilings of $\HH$ (one of them is shown on fig.~\ref{5}).
On the other hand, Lalley in paper \cite{Lal} proves similar facts about Bernoulli percolation for Cayley graphs of cocompact Fuchsian groups of genus at least $2$ and for some class of hyperbolic triangle groups (namely: groups of presentation $\langle c_1,c_2,c_3|c_1^2=c_2^{4m}=c_3^{4m}=c_1c_2c_3=1\rangle$, where $m\ge 5$).
%, which gives a bit smaller class of graphs than \vttgs{}.
%For example graphs of regular tilings all belong to \vttgs{} (not exhausting the class) but that ones of odd degree are not considered in \cite{Lal}, because they cannot be Cayley graphs. On the other hand the class considered there by Lalley is wider than that of graphs of regular tilings with even degree.
\end{rem}

\begin{figure}
  \center
  \includegraphics[width=0.5\textwidth,bb=0 0 341 341]{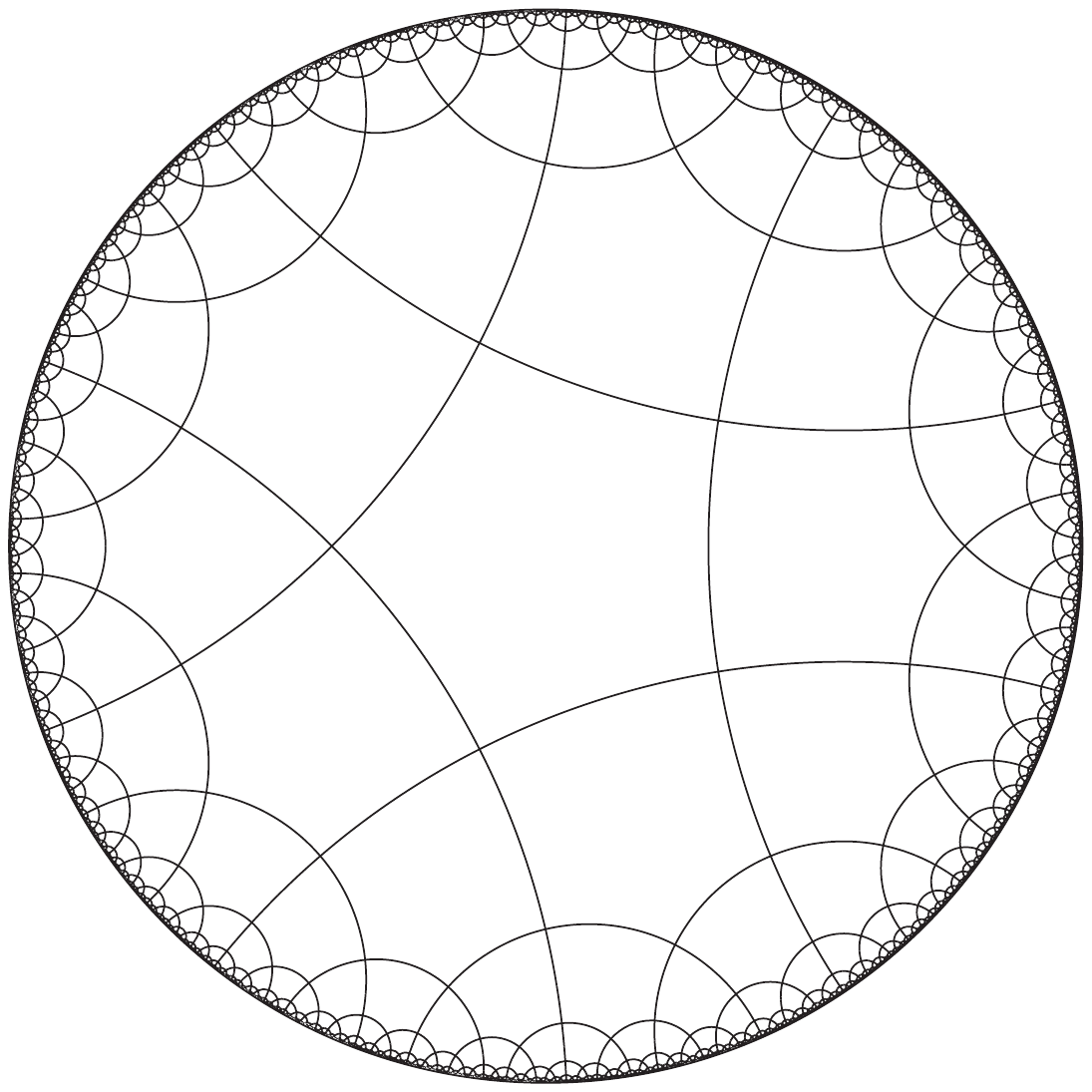}
  \caption{An example of tiling of $\HH$ by regular right-angled pentagons.\label{5}}
\end{figure}

First of all I prove that, indeed, on the graphs I consider, there are three essential phases of Bernoulli bond percolation, mentioned in the introduction. Before that, I define some properties of graphs:

\begin{df} \label{grprop}
Let $G$ be any locally finite graph. We define it to:
\begin{itemize}
%\item jest tranzytywny, gdy jego grupa automorfizm\'ow $\mathrm {Aut}(G)$ dzia{\l}a tranzytywnie na jego zbiorze wierzcho{\l}k\'ow.
\item \textbf{have one end}, if for any finite set $V_0\subseteq V(G)$ the subgraph induced on its complement $V(G)\sm V_0$ has exactly one unbounded component.
\item be \textbf{nonamenable}, if there is a constant $\varepsilon>0$ such that for every finite $V_0\subseteq V(G)$ we have $|\partial V_0|\ge\varepsilon|V_0|$, where $\partial V_0$ is the set of edges of $G$ with exactly one vertex in $V_0$. Otherwise we call $G$ \textbf{amenable}.
\end{itemize}
One defines also \textbf{edge isoperimetric constant} of $G$:
$$\Phi(G)=\inf\{\frac{|\partial V_0|}{|V_0|}:\emptyset\neq V_0\subseteq V(G)\text{ -- finite}\}.$$
Note that $G$ is nonamenable iff $\Phi>0$.
\end{df}

\begin{thm}
For any \vttg{} $G$ we have
$$0<\pc(G)<\pu(G)<1.$$
%and the number of \ics{} in $\op$ is \as{} equal to:
%\begin{itemize}
%\item $0$, when $0\le p\le\pc$;
%\item $0$, when $\pc< p\le\pc$;
%\item $1$, when $\pu\le p\le 1$;
%\end{itemize}
\end{thm}

\begin{proof}{}
Basing on the following theorem from \cite{BS} (theorem 1.1 there), it is enough to prove the assumptions of it about $G$:

\begin{thm} \label{1.1}
Let $G$ be a transitive, nonamenable, planar graph with one end. Then $$0<\pc(\G_A)<\pu(\G_A)<1,$$ for Bernoulli bond or site\footnote{Bernoulli \textbf{site} percolation is performed by removing vertices of the graph (instead of edges in bond percolation).} percolation on $G$.
\end{thm}
Planarity and transitiveness are obviously satisfied, so the remaining properties of $G$ to show are having only one end and nonamenability:
\begin{itemize} \item One end: \end{itemize}
\par Let $V_0$ be finite subset of $V(G)$. Remove $V_0$ from $V(G)$ and take the induced subgraph $G'$ (here I mean the plane graph). Take a \hyp{} ball $B$ which covers $V_0$ together with all the tiles meeting $V_0$. Now, for every two vertices not lying in $B$ there is a polygonal path $P_0$ in $\HH$ joining them and not intersecting $B$. We can replace that path by a path $P$ in graph $G$ chosen to go along perimeters of consecutive tiles visited by $P_0$. That path may meet $B$, but is still disjoint with $V_0$. Hence all vertices in $V(G)\sm B$ lie in one component of $G'$. But the rest of vertices of $G$ lie in $B$, so there are finitely many of them, whence $G'$ has exactly one unbounded component.
\begin{itemize} \item Nonamenability: \end{itemize}
We need to introduce some notions:
\begin{df}
Let $(X,d_X)$, $(Y,d_Y)$ be arbitrary metric spaces. A \textbf{\qi} between $X$ and $Y$ is a (not necessarily continuous) map $f:X\to Y$ s. t. there exist constants $A\ge 1$ and $B\ge 0$ satisfying for all $x,y\in X$
$$\frac{1}{A}d_X(x,y)-B\le d_Y(f(x),f(y))\le Ad_X(x,y)+B$$
and, moreover, for any $y\in Y$ the distance of $y$ from $\mathrm{Im}f$ does not exceed $B$.
\par Let group $\G$ act by isometries on a metric space $X$. We say that this action is \textbf{proper} if for each $x\in X$ there exist $r>0$ such that the set $\{\g\in\G:B(x,r)\cup\g B(x,r)\neq\emptyset\}$ is finite, where $B(x,r)$ denotes metric ball in $X$ of origin $x$ and radius $r$. For isometric action of $\G$ on $X$ also, we call that action \textbf{cocompact} if the orbit of some compact subset of $X$ covers $X$.
\par One more notion will be convenient to use: for a tiling of $\HH$ and its vertex $x$, the \textbf{star} of $x$ is the sum of all tiles containing $x$. Its interior is called \textbf{open star} of $x$.\footnote{The general definition of star (e.~g. for simplicial complexes) has a bit different form, see for example definition 7.3, chapter I.7 in \cite{BH}.}
\end{df}
It is straightforward to check that nonamenability of graphs is invariant on quasi-isometry of graphs (with usual graph metric). (To prove that, one can use an alternative, but equivalent definition of nonamenability dealing with the set $\mathcal{N}_r(V_0)\sm V_0$ instead of $\partial V_0$ from above definition, where $\mathcal{N}_r(V_0)\subseteq V(G)$ is neighbourhood of radius $r$ of given finite set $V_0$ of vertices.)
\par So it is sufficient to show that $G$ is quasi-isometric to some nonamenable graph. That comparison graph will be the graph of regular tiling of $\HH$ by regular pentagons, five of them meeting in each vertex of the tiling. Let us call this graph  $G_{\{5,5\}}$.\footnote{$\{5,5\}$ is so-called Schl\"afli symbol of that regular tiling.} It is indeed nonamenable, because basing on \cite{AP}, theorem 4.1, the edge isoperimetric constant $\Phi(G_{\{5,5\}})$ can be calculated as $\sqrt{5}$, which is strictly positive. The \qi{} will be shown in a couple of steps. First, $G$ is \qic{} to some group $\G$ of isometries of $\HH$ acting transitively on $V(G)$. It is so because the action of such $\G$ on $G$ is proper and cocompact\footnote{In other words: \textbf{geometric}.} (considered with graph metric on $G$ -- not only on $V(G)$). The properness, roughly speaking, follows from finiteness of the subgroups fixing any vertex, and cocompactness -- from transitivity. Hence, by the \v{S}varc-Milnor Lemma (stated in \cite{BH}, chapter I.8, as prop. 8.19) $\G$ is finitely generated and $\G$ and $G$ are \qic{}.\footnote{Here $\G$ considered with the word metric, which -- up to \qi{} -- does not depend on the choice of finite generating set of $\G$.} Similarly, $\G$ is \qic{} to the \hyp{} plane $\HH$ itself: the action of $\G$ on $\HH$ is proper (take a ball included in the open star of a vertex in the tiling which contains given point of $\HH$) and cocompact (take the star of a vertex). In that way we showed a \qi{} between $G$ and $\HH$ (using transitiveness of \qi{}). In particular, it is true when we take $G=G_{\{5,5\}}$, so in our setting also $G$ and $G_{\{5,5\}}$ are \qic{}, as we desired. Hence $G$ is nonamenable.
The above completes the proof of the theorem.
\end{proof}

\begin{rem} \label{pureph}
In fact, for $p\in[0;\pc]$ there are \as{} no \imic{} in $\op$, there are \as{} $\infty$ of them for $p\in(\pc;\pu)$ and exactly $1$ for $p\in[\pu;1]$ (so we have three essential and pure phases, determined by the number of \ics{}). The same is true about the dual $G^\dag$ (see section \ref{secdual} for notions of duality).
Those remarks can be easily deduced from theorem 1.1, 3.7 and 1.3 of \cite{BS} (see also proofs of theorems 1.1 and 3.8 there; the fact that the event of existence of an infinite cluster is increasing should be used; for increasing and decreasing events, see \cite{grim}, chapter 2.1, especially theorem 2.1).
\end{rem}

\begin{rem} \label{clgr}
One can easily deduce from the proof of proposition 2.1 from \cite{BS} that in fact any transitive, nonamenable, planar graph with one end can be realized as a \vttg{} in $\HH$. Hence \vttgs{} are all the graphs known by \cite{BS} to have three essential phases of Bernoulli bond percolation.
\par It turns out also that, in that setting, the property that all the \ics{} of the random subgraph have one-point boundaries of ends, does not depend on the embedding of the underlying whole graph in $\HH$, but just on the abstract graph. This can be explained in terms of Gromov boundary\footnote{For basics on Gromov boudaries, see \cite{BH}, chapter III.H.3.}: $\bd\HH$ can be defined as the Gromov boundary of $\HH$. On the other hand, when graph $G$ is embedded by a \qi{} in $\HH$ (it is then closed in $\HH$), then by \cite{BH}, chapter III.H, theorem 1.9, $G$ is hyperbolic (in the sense of Gromov) and from theorem 3.9 from that chapter that \qi{} induces a homeomorphism of the Gromov boundaries of $G$ and $\HH$. Let $\hat{G}$ be the compactification of abstract graph $G$ by its Gromov boundary. Then one can embed $\hat{G}$ in $\cH$ sending $\bd G$ onto $\bd\HH$ by above homeomorphism. One can easily check that for a subset $A$ of $G$ its ends and boundaries of ends are in principle the same as when $A$ is considered a subset of above embedding in $\HH$. It follows that phenomenon of $1$-dimensional clusters occuring on abstract $\hat{G}$ and on $\cH$ agree.
\end{rem}

\section{Main theorem}

Before I prove the main theorem (theorem \ref{main}), I need following lemmas.
\par Let $G$ be a \vttg{} and $\op$ be $p$-Bernoulli bond percolation process on $G$ in the middle phase.

\begin{lem} \label{D}
The limits in $\bd\HH$ of paths in $\op$ a.~s.~lie densely in $\bd\HH$.
\end{lem}
%I will call this property of $\op$ and $\opd$ \textbf{(D)}.

\begin{proof}{}
This can be deduced from the theorem 4.1 and lemma 4.3 of \cite{BS}, which I quote here:

\begin{thm} \label{BSlim}
Let $T$ be a vertex-transitive tiling of $\HH$ with finite sided faces, let $G$ be the graph of $T$, and let $\omega$ be Bernoulli percolation on $G$. Almost surely, every infinite component of $\omega$ contains a path that has a unique limit point in $\bd\HH$.
\end{thm}
The following lemma is formulated for \textbf{invariant percolation process} on $G$, i.~e. random subgraph process whose probability distribution is invariant on some vertex-transitive group action on $G$. Bernoulli bond percolation is an example of invariant percolation.
\begin{lem} \label{BSdns}
Let $T$ be a vertex-transitive tiling of $\HH$ with finite sided faces, let $G$ be the graph of $T$, and let $\omega$ be invariant percolation on $G$. Let $Z$ be the set of points $z\in\bd\HH$ such that there is a path in $\omega$ with limit $z$. Then \as{} $Z=\emptyset$ or $Z$ is dense in $\bd\HH$.
\end{lem}

Basing on theorem \ref{BSlim} and on remark \ref{pureph}, in our situation there are \as{} some paths in $\op$ with limit points in $\bd\HH$ and hence set $Z$ from lemma \ref{BSdns} is \as{} dense in $\bd\HH$.
\end{proof}

\begin{rem}
In special case of $G$ being the graph of regular tiling of $\HH$ with right-angled pentagons and $p>\frac{1}{2}$, this lemma can also be proved in the following more elementary way, similar to the technique used in proof of theorem 1 in \cite{Lal} (on p. 171):
\par I embed an infinite complete binary tree in the graph $G$ (see fig.~\ref{rys_drz}). (It is done using \hyp{} geometry.)
\begin{figure}[t]
  \center
  \includegraphics[width=0.5\textwidth,bb=0 0 462 462]{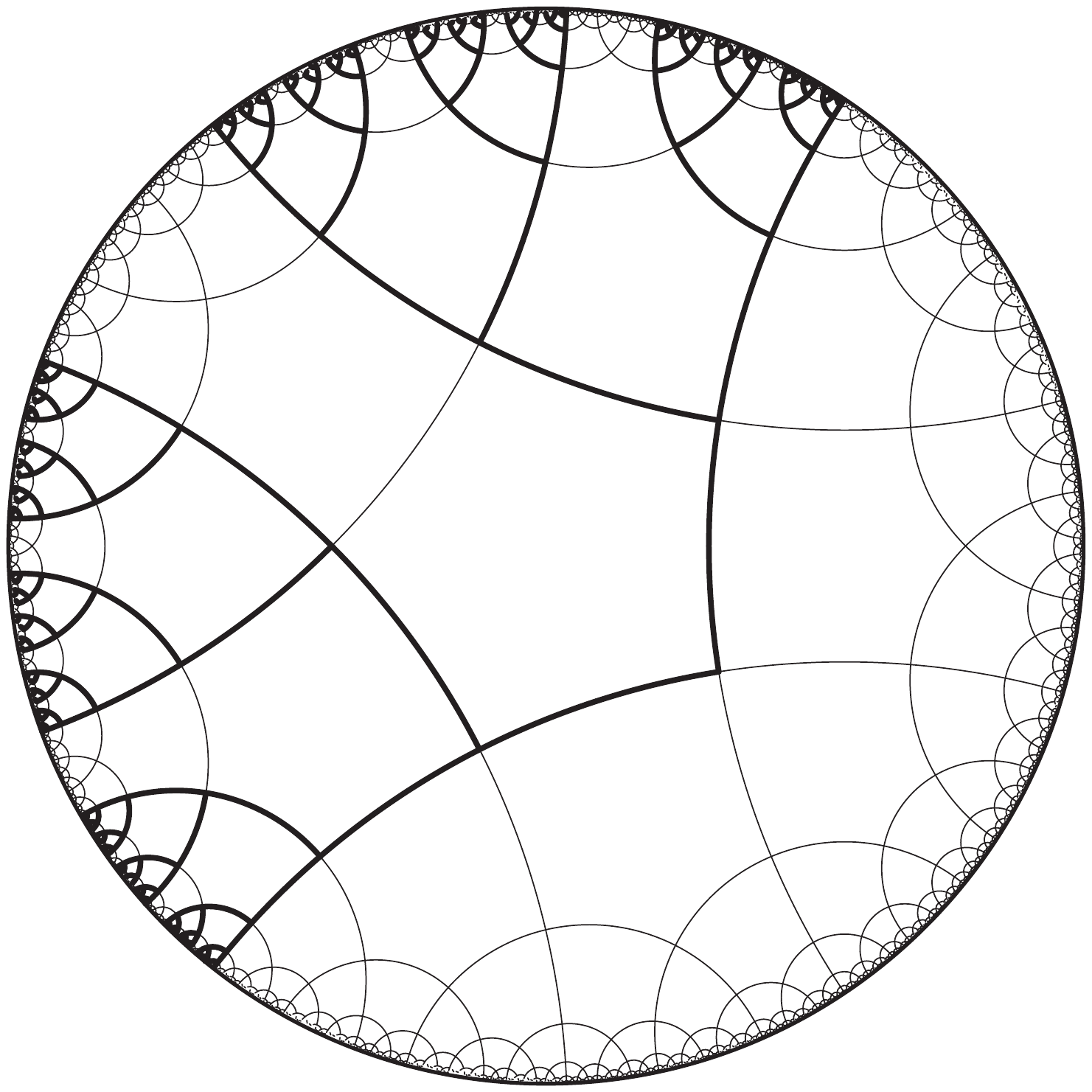}
  \caption{Infinite complete binary tree embedded %in $G$ (on the left) and
in $G$ %(on the right)
in lemma \ref{D}.\label{rys_drz}}
\end{figure}
\par When I have such a tree $T$ embedded in $G$, I can move it by an isometry $\g$ preserving $G$ so that $\bd\g(T)$ will be included in arbitrary (small) arc $\Phi$ of $\bd\HH$ (see proposition \ref{H2H}). The random graph $\op\cap\g(T)$ is $p$-percolation process on the tree $\g(T)$, where the critical probability equals $\frac{1}{2}$. So for $p>\frac{1}{2}$ we a.~s.~obtain an open infinite path in $\op\cap\g(T)$ with limit in $\bd\g(T)\subseteq \Phi$. Such limits lie a.~s.~densely in $\bd\HH$.\hfill\rule[0.025cm]{0.21cm}{0.21cm}
%\par The case of $G^\dag$ is done in the same way, noting that $\opd$ is just \linebreak $(1-p)$-percolation on $G^\dag$.
\end{rem}

\begin{lem} \label{H}
In the middle phase a.~s.~every halfplane meets \imic{} of $\op$.
\end{lem}
%Let's call this graph property \textbf{(H)}.

\begin{rem}
In this paper a \textbf{halfplane} is always closed.
\end{rem}

Before the proof of the lemma let us consider a group $\G$ of isometries of $\HH$ which acts transitively on vertices of $G$ (by the assumption on $G$). One can easily see that $\G$ is a discrete subgroup of $\Isom(\HH)$ (with the usual topology), because it preserves a tiling of $\HH$. Basing on that we are going to say something about the action of $\G$ on $\HH$ using basic theory of Fuchsian groups, which can be found in \cite{K}.

\begin{df}
There are three kinds of orientation preserving isometries of $\HH$ other than identity: so-called hyperbolic, parabolic and elliptic. That classification is based on how many fixed points in $\bd\HH$ has such isometry (it makes sense, since every isometry of $\HH$ extends continuously in a unique way to a homeomorphism of $\cH$). Such isometries with exactly two fixed points in $\bd\HH$ are \textbf{hyperbolic}, one fixed point -- \textbf{parabolic} and no fixed points -- \textbf{elliptic}. One may think of hyperbolic and elliptic isometries as of analogues of translations and rotations in Euclidean plane, respectively. Some basics of these notions are present in \cite{K}.
\par A \textbf{Fuchsian group} is discrete subgroup of $\Isom(\HH)$ consisting only of orientation preserving isometries of $\HH$.
The \textbf{limit set} of a Fuchsian group $\GF$ is the boundary $\bd\GF x_0$ of orbit $\GF x_0$ of some point $x_0\in\HH$ (one can observe that it does not depend on the choice of $x_0$).
\end{df}
Let $\GF$ be the Fuchsian group of all orientation preserving isometries in $\G$. (The index of this subgroup in $\G$ is at most $2$.) We claim that $\GF$ acts cocompactly on $\HH$. Indeed, since $\G$ itself acts cocompactly on $\HH$, which means that there exists a compact set $K\subseteq\HH$ s.~t. the family $\G K$ covers $\HH$, then if we take $K\cup \g K$, where $\g$ is some orientation changing isometry $\g\in\G$, we have covering of $\HH$ by $\GF(K\cup \g K)$.
\par Next we observe that the limit set of $\G$ is the whole $\bd\HH$. If it were not, then some halfplane would be disjoint with some orbit of a point in $\GF$, which is impossible because of cocompactness of $\GF$. So, by theorem 3.4.4 from \cite{K}, the set of fixed points in $\bd\HH$ of hyperbolic translations is dense in $\bd\HH$.
\par That gives us the following fact:

\begin{prop} \label{H2H}
Every halfplane $H_1$ in $\HH$ can be mapped into any halfplane $H_2$ by some isometry in $\GF$ (and hence in $\G$).
\end{prop}
\begin{proof}{}
Take arbitrary halfplanes $H_1$ and $H_2$. Let $\g\in\GF$ be a \hyp{} translation with attracting point $a_\g$ lying in the interior of the closed arc $\bd H_2$. If the repelling point $r_\g$ of $\g$ is not in $\bd H_1$, then some multiply composition of $\g$ moves $H_1$ into $H_2$.
(see left picture on fig.~\ref{figH2H}).
Now if $r_\g\in\bd H_1$, then take any $\d\in\GF$ with repelling point $r_\d$ distinct from $a_\g$ and $r_\g$ and not lying in $\bd H_1$. It is clear from the proof of theorem 2.4.3 of \cite{K} that the attracting point $a_\d$ of $\d$ is as well different from $r_\g$. Hence again some multiply composition of $\d$ maps $H_1$ to $H_1'$ which is arbitrarily close to $a_\d$, so that its boundary $\bd H_1'$ does not include point $r_\g$.
(middle picture on fig.~\ref{figH2H}).
Then some multiply composition of $\g$ pushes $H_1'$ into $H_2$.
(see right picture on fig.~\ref{figH2H}).
Composition of these two compositions gives us desired isometry.
\begin{figure}[t]
  \center
\includegraphics[width=0.3\textwidth,bb=0 0 432 427]{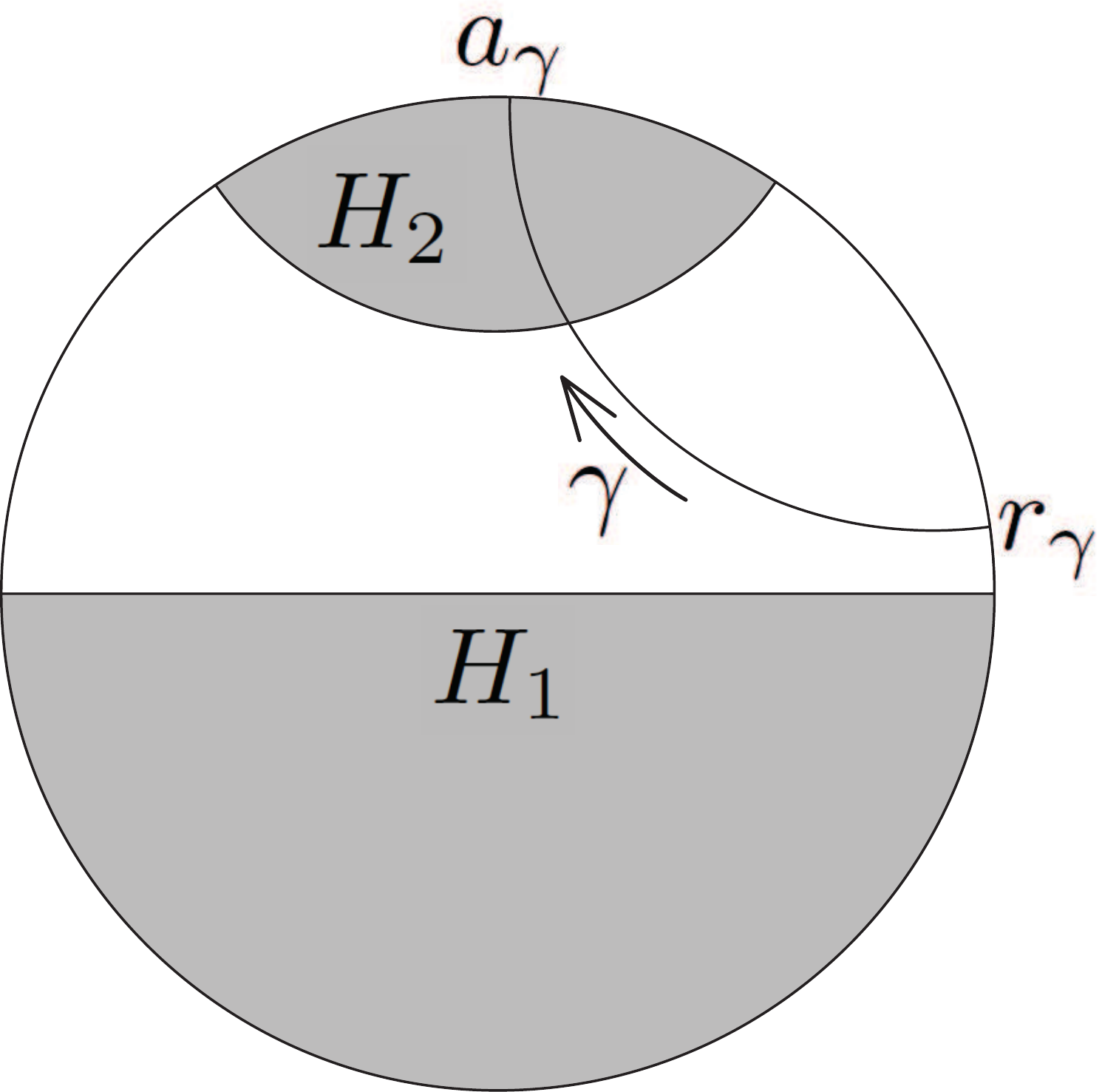}\qquad
\includegraphics[width=0.3\textwidth,bb=0 0 416 445]{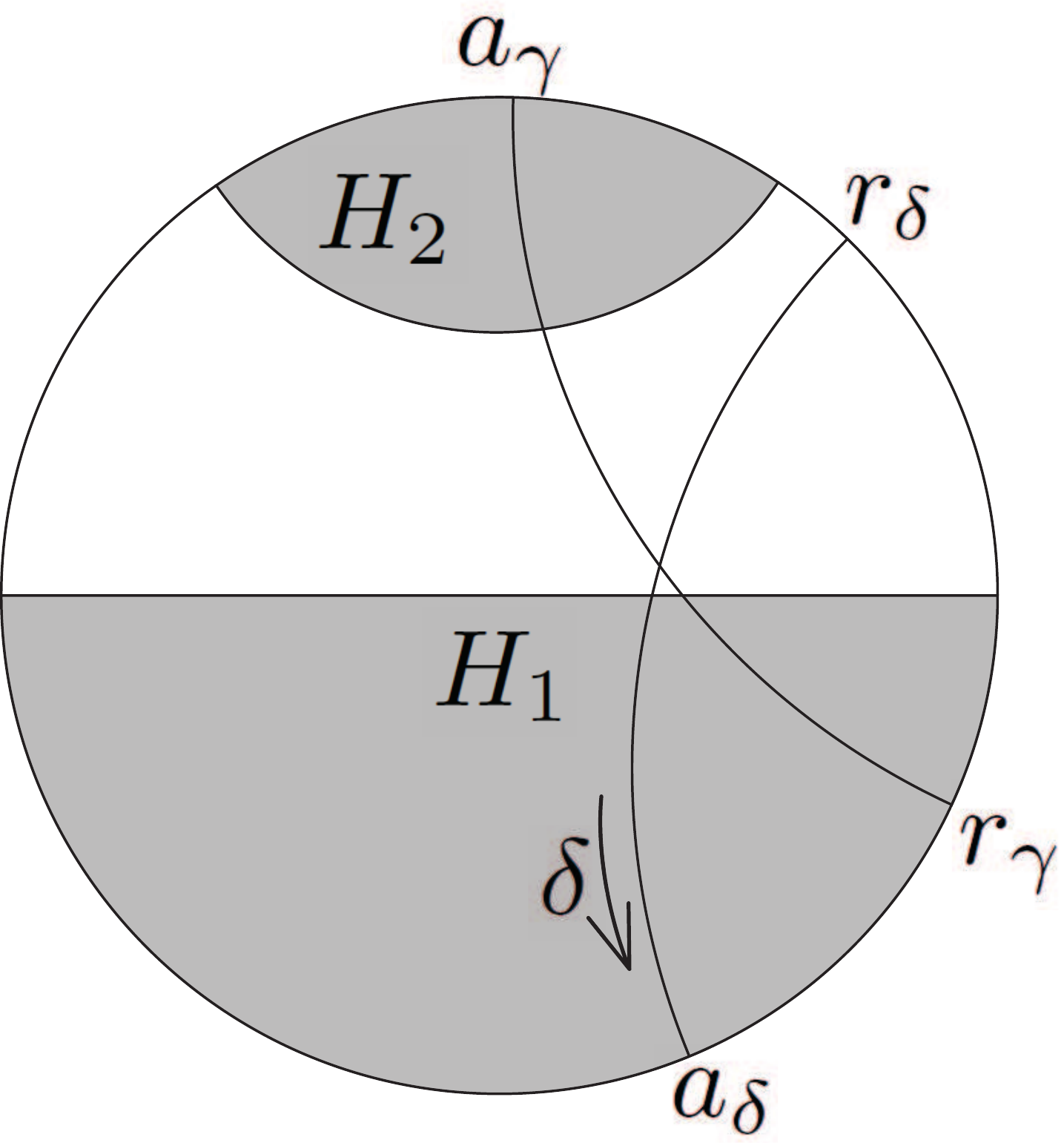}
\includegraphics[width=0.3\textwidth,bb=0 0 416 445]{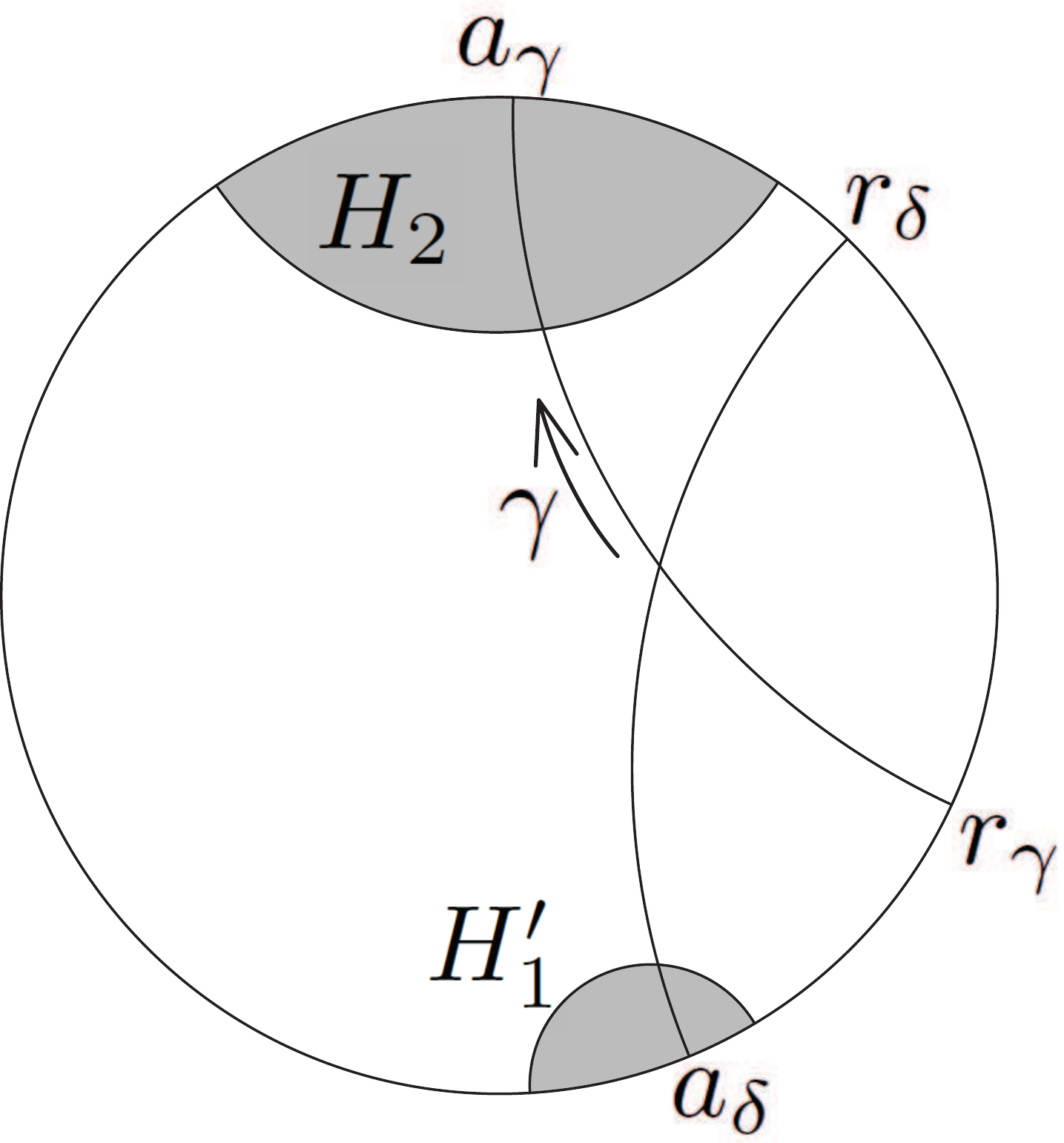}
  \caption{Proof of proposition \ref{H2H}.\label{figH2H}}
\end{figure}
\end{proof}

\begin{proof}{of the lemma}
Let us assume \emph{a contrario} that there is a halfplane $H$ which meets only finitely many \ics{} of $\op$ with positive probability. In such situation the halfplane $H'=\cl{H^\c}$ includes entirely infinitely many \ics{} (by remark \ref{pureph}). Let $H_1,H_2,\ldots$ be a sequence of pairwise disjoint halfplanes all lying in $H$, and even more: such that distances between them are greater than twice the maximal hyperbolic length of an edge in $G$ (see the fig.~\ref{infm H}). By the above proposition we can move $H'$ by some sequence of isometries $\g_1,\g_2,\ldots\in\G$ into $H_1,H_2,\ldots$, respectively.

\begin{figure}[t]
  \center
  \includegraphics[width=0.4\textwidth,bb=0 0 404 404]{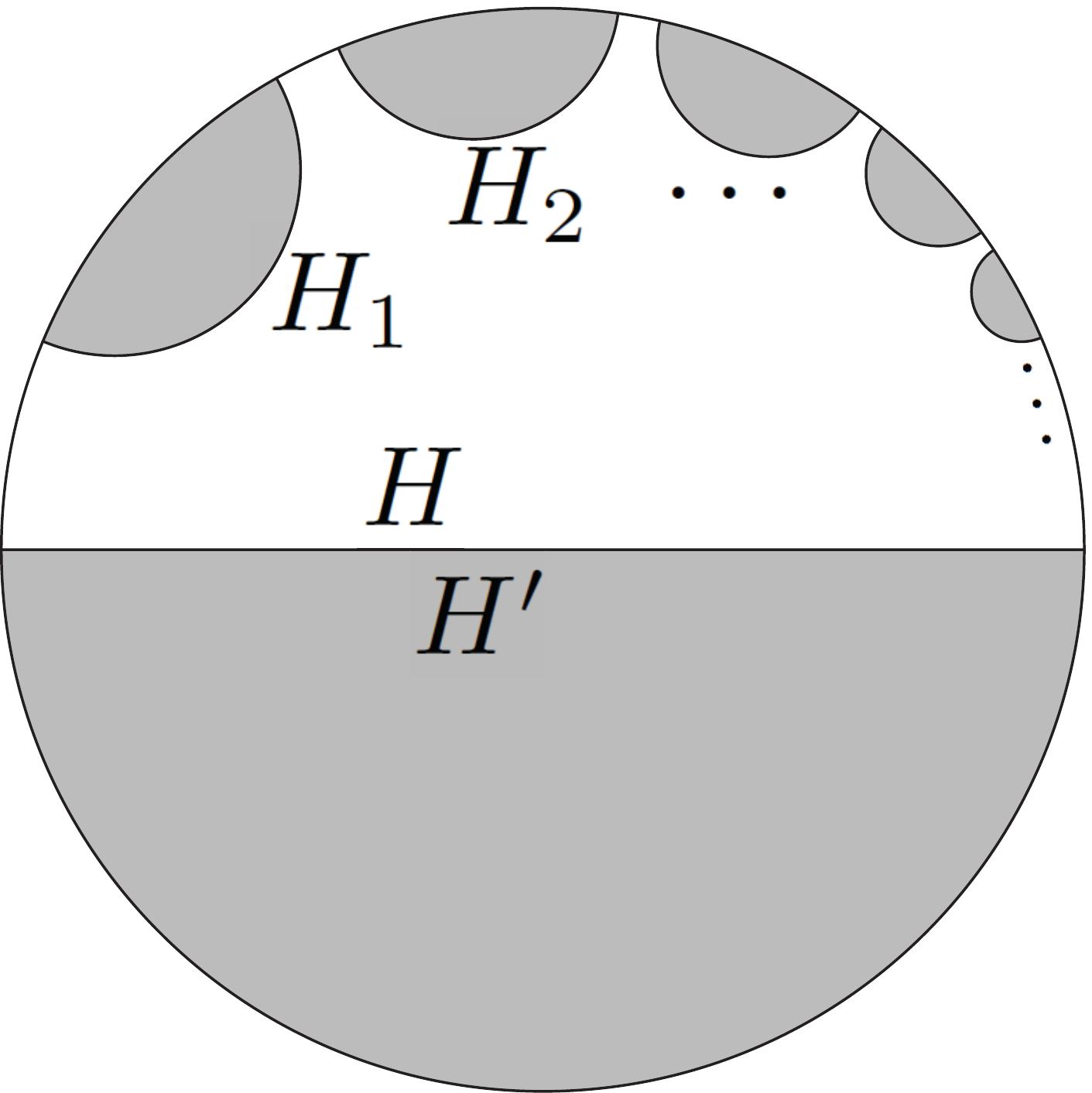}
  \caption{Proof of lemma \ref{H}.\label{infm H}}
\end{figure}
Note that one can precisely say whether a halfplane contains \imic{} looking only on the behaviour of $\op$ on the edges intersecting with that halfplane. So the random event $C(I)$ that in a halfplane $I$ there are \imic{} depends only on those edges, for any halfplane $I$. There follows that events $C(H_1), C(H_2),\ldots$ are stochastically independent, because the underlying sets of edges are pairwise disjoint. Moreover, they have the same positive probability as $C(H')$, so the probability that none of them occurs is less or equal than $(1-\Pr(C(H')))^n$ for any $n\in\mathbb{N}$, whence equal to $0$. That gives us that some $H_n$ \as{} contains \imic{} but so does $H$, because it includes $H_n$, a contradiction. That ends the proof of the lemma.
\end{proof}

Now I state the main theorem:
\begin{thm} \label{main}
In the middle phase of Bernoulli bond percolation on any \vttg{} $G$ a.~s.~all the ends of all the \ics{}
have one-point boundaries in $\bd\HH$.
\end{thm}
\begin{proof}{}
The techniques used here are similar to those of Lalley used in \cite{Lal}.
Let $\op$ be $p$-Bernoulli bond percolation process on $G$, when $p\in(\pc(G);\pu(G))$. Let us assume \emph{a contrario} that with positive probability there is an end $e$ of an infinite cluster $a$ of $\op$ with non one-point boundary.
%So with the same probability it is the case with the assumption that the graph $\op$ has properties (D) and (H) (because that is satisfied a.~s.).
\par One can prove a topological fact saying that always the boundary of an end is connected and compact (the proof is given in Appendix). So in our situation $\bd e$ is a non-degenerate closed arc in $\bd\HH$, or the whole $\bd\HH$. Let $\Phi$ be an open non-epty arc in $\bd\HH$, included in $\bd e$. By lemma \ref{D} the limits of paths in $\op$ lie densely in $\Phi$. I consider two cases:

\begin{itemize} \item  There are two paths $P_1, P_2\subseteq\op$ not lying in $a$ with distinct limits in $\Phi$. \end{itemize}
\par Let us take a closed ball $B$ in $\HH$, meeting $P_1$ and $P_2$. Then $\bd e(B)$ (and also $\clcH{e(B)}$) contains $\Phi$, and $e(B)$ is connected, but $P_1$ and $P_2$ have limits in $\Phi$ so they should cut $e(B)\subseteq a$, which is a contradiction (see the picture \ref{geom1}).
\begin{figure}[h]
  \center
  \includegraphics[width=0.6\textwidth,bb=0 0 417 437]{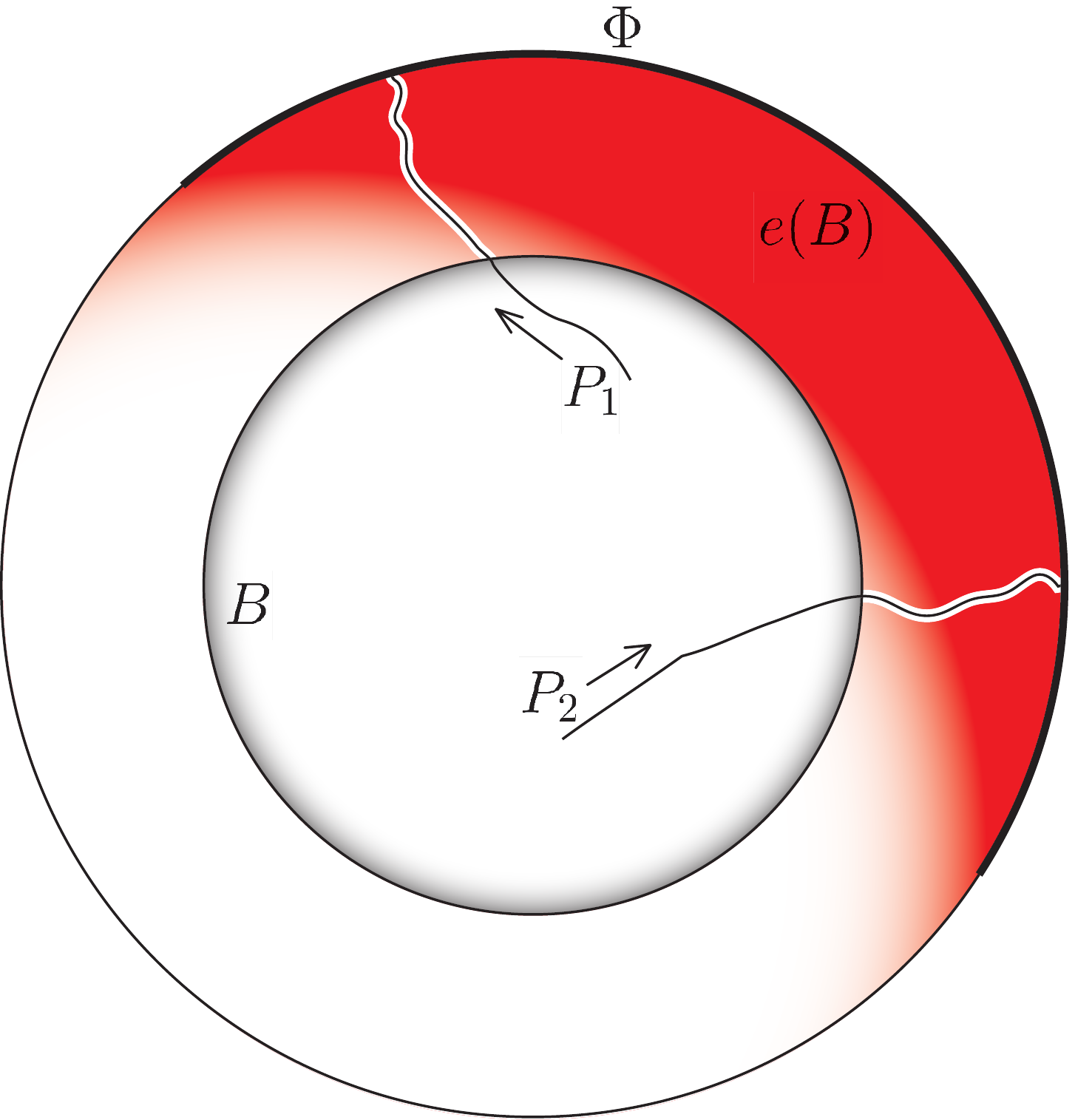}
  \caption{Proof of theorem \ref{main}, the first case.\label{geom1}}
\end{figure}
%Then, similarly as in case \ref{dual}, we have contradiction, because $e(B)\subseteq a$ is connected and $\clcH{e(B)}$ contains $\Phi$ (with limits of $P_1$, $P_2$, so these paths need to cut $e(B)$).

\begin{itemize} \item In $\Phi$ there are infinitely many limits of open paths lying in $a$. \end{itemize}
Then, let us take two such paths $P_1$, $P_2$ with distinct limits in $\Phi$ and two others $P_1'$, $P_2'$ with still other limits in $\Phi$ such as in fig.~\ref{geom22}.
\par We can join $P_1$ and $P_2$ by an open path $P_0$ in $a$ and so $P_1'$ and $P_2'$ by $P_0'$ in $a$. It provides paths $\sigma, \sigma'\subseteq a$ shown in fig.~\ref{geom22}, which disconnects $\HH$ into components, two of which -- $C$ and $D$ -- are shown in the figure. We can take two halfplanes lying in $C$ and $D$, resp. From lemma \ref{H} we know that each of them a.~s.~meets some infinite cluster other than $a$. So one of these clusters lies in $C$ and the other in $D$ -- denote them $c$ and $d$, respectively. So $\bd c\subseteq\bd C$ and $\bd d\subseteq\bd D$, which means that for a sufficiently large ball $B$ the union of $c$ and $d$ disconnects $\HH\setminus B$ into components, two of which are $S_1$ and $S_2$ containing resp.~the tails of $P_1$, $P_1'$ and $P_2$, $P_2'$. But the areas of $S_i$ between $P_i$ and $P_i'$ for $i=1,2$ meet $e(B)$ (because their boudaries lie in $\Phi$) so $e(B)$ meet both $S_1$ and $S_2$, which means that $e(B)$ is not connected (because it is disjoint with $c$ and $d$), a contradiction.

\begin{figure}[t]
  \center
  \includegraphics[width=0.6\textwidth,bb=0 0 417 440]{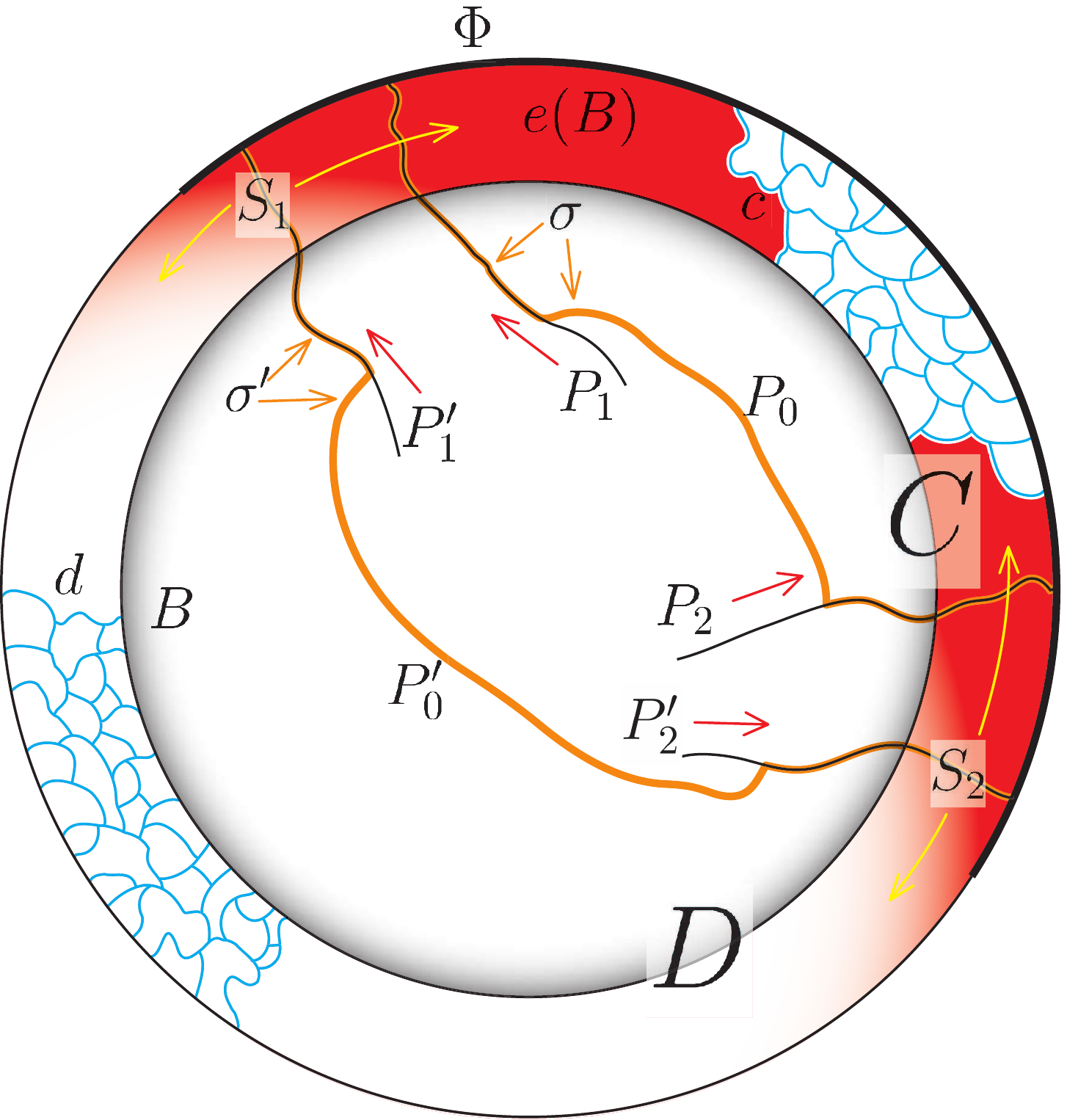}
  \caption{Proof of theorem \ref{main}, the second case.\label{geom22}}
\end{figure}

\par This ends the proof.
\end{proof}

\section{Dual graphs} \label{secdual}

\begin{cor} \label{dualmain}
Theorem \ref{main} also applies to the dual graph of any \vttg{} $G$.
\end{cor}

Let us introduce notions of duality:
\begin{df} \label{dfdual}
For any plane graph $G$ one defines its \textbf{dual graph} $G^\dag$: the set of vertices is the set of faces of $G$ and two such vertices are joined by an edge, iff the corresponding faces are neighbours by an edge in $G$. (Note that to define the dual graph the plane realization is needed, not only the abstract graph.)
Such dual graph is also a planar graph, because one can realize it in the plane placing its vertices inside the faces of the original graph $G$ (called also the \textbf{primal graph}), and constructing the edges as some plane paths leading from any vertex of the dual graph to an interior point of an edge of the face including it, then to the vertex inside the second face touching that edge. We call the constructed edge the \textbf{dual edge} to the original edge, which is cut by it in exactly one point.
\end{df}
\begin{rem}
The dual graph of a plane graph $G$ does not need to be a simple graph. (But in our situation it is.)
\end{rem}

I will consider the dual percolation process for $\op$, which I define below:
\begin{df}
For a plane graph $G$ of a polygonal tiling and for any edge $e$ of $G$ let $e^\dag\in G^\dag$ denote the dual edge of $e$ in the dual graph $G^\dag$. (The operation $e\mapsto e^\dag$ is a bijection between $E(G)$ and $E(G^\dag)$.) Now let for any subgraph $H$ of $G$ the ,,dual subgraph'' $H^\dag$ be the subgraph of $G^\dag$ such that $V(H^\dag)=V(G^\dag)$ and $E(H^\dag)=\{e^\dag:e\in E(G)\sm E(H)\}$.
\par Then the random subgraph $\opd$ (dual to $\op$) is called the \textbf{dual percolation process} (dual to $\op$).
\end{df}
\begin{rem} \label{uw_dual}
Notice that $\opd$ is in fact a $(1-p)$-Bernoulli bond percolation on $G^\dag$
\end{rem}

%(The \textbf{dual graph} $G^\dag$ of a plane graph $G$ is obtained by taking all faces of $G$ for vertices of $G^\dag$ and joining them with an edge in $G^\dag$ iff the corresponding faces are neighbours.)
\begin{proof}{of the corollary}
For given \vttg{} $G$ and its dual $G^\dag$, use the fact that in the middle phase percolation on both the graphs we have \imic{} (see remark \ref{pureph}). Then in setting of proof of theorem \ref{main} (with $\Phi$, $a$ and $e$), but with $a$ -- component of $\opd$ instead of $\op$, we know by lemma \ref{D} that the limits of paths in $\op$ lie densely in $\Phi$. So there are two paths $P_1, P_2\subseteq\op$ with distinct limits in $\Phi$. Then, similarly as in the first case in proof of the theorem, we have contradiction, because $e(B)\subseteq a$ is connected and $\clcH{e(B)}$ contains $\Phi$ (with limits of $P_1$, $P_2$, so these paths need to cut $e(B)$). See figure \ref{geom1}.
\end{proof}

\section*{Appendix}%\app...

In this appendix I consider a topological space $X$ as in definition \ref{dfbd} (i.~e. locally compact and $\mathrm T_{3\frac{1}{2}}$) together with some compactification $\cX$ of it.
\begin{rem}
Recall that $\bd X=\cX\sm X$ is always closed (and $X$ is open) in $\cX$.
\par There is topological notion of boundary (with other meaning than $\bd$ in definition \ref{dfbd}). Due to it I will call that notion \textbf{topological boundary}.
\par Now let us consider a set $A\subseteq X$ and its arbitrary end $e$. Notice that then for any compact $K\subseteq X$ the set $\bd e(K)$ is compact (as subspace of $\cX$).
\par It is worth noting that in the above setting $\bd e(K)\neq\emptyset$. It is so because $e(K)$ is not conditionally compact in $X$; if it were, $\clX{e(K)}$ would be compact and $e(K\cup\clX{e(K)})\subseteq e(K)$, but $e(K\cup\clX{e(K)})$ and $\clX{e(K)}$ are disjoint (by the definition of end), so $e(K\cup\clX{e(K)})=\emptyset$, which contradicts the definition of component.
\par Similarly, $\bd e$ itself is non-empty as an intersection of family of compact sets from the definition, whose each finite subfamily has, by an easy excercise, non-empty intersection.
\end{rem}
\begin{lem} \label{sp}
For any topological space $X$, which is locally compact and $\mathrm T_{3\frac{1}{2}}$ (as above) and for any compactification $\cX$ of it and for any $a\subseteq X$ every end $e$ of $a$ has non-empty connected boundary.
\end{lem}
\begin{proof}{}
The set $\bd e$ is non-empty by the above remark, so it remains to show the connectivity.
\par Let us assume \emph{a contrario} that $\bd e$ is not connected. Then it is a sum of two closed disjoint non-empty sets $C,D\in\cX$:
$$\bd e= C\cupdot D.$$
Because $\cX$ is normal, there exist disjoint open neighbourhoods $U$ and $V$ in $\cX$ of the sets respectively $C$ and $D$.
\begin{claim}
There is a compact set $K\subseteq X$ such that $\bd e(K)\subseteq U\cupdot V$.
\end{claim}
\begin{proof}{}
Let us consider the family $\{U\cup V,(\bd e(K))^c:K\subseteq X, K\text{ -- compact}\}$. It is an open cover of $\cX$, because
$$\cX = (U\cup V)\cup(\bd e)^c = (U\cup V)\cup\bigcup_{\substack{K\subseteq X\\ K\text{ -- compact}}}(\bd e(K))^c.$$
Hence there is a finite subcover $\{U\cup V, (\bd e(K_1))^c,\ldots,(\bd e(K_n))^c\}$ for some compact $K_1,\ldots,K_n\subseteq X$.
Let us take $K=\bigcup_{i=1}^n K_i$. Then
$$\bd e(K)\subseteq \bigcap_{i=1}^n\bd e(K_i)\text{,\quad so\quad}\bigcup_{i=1}^n(\bd e(K_i))^c\subseteq(\bd e(K))^c$$
and $\{U\cup V,(\bd e(K))^c\}$ is also a cover of $\cX$. Hence $\bd e(K)\subseteq U\cupdot V$ as we desired.
\end{proof}

\begin{figure}
  \center
  \includegraphics[width=.6\textwidth,bb=0 0 357 386]{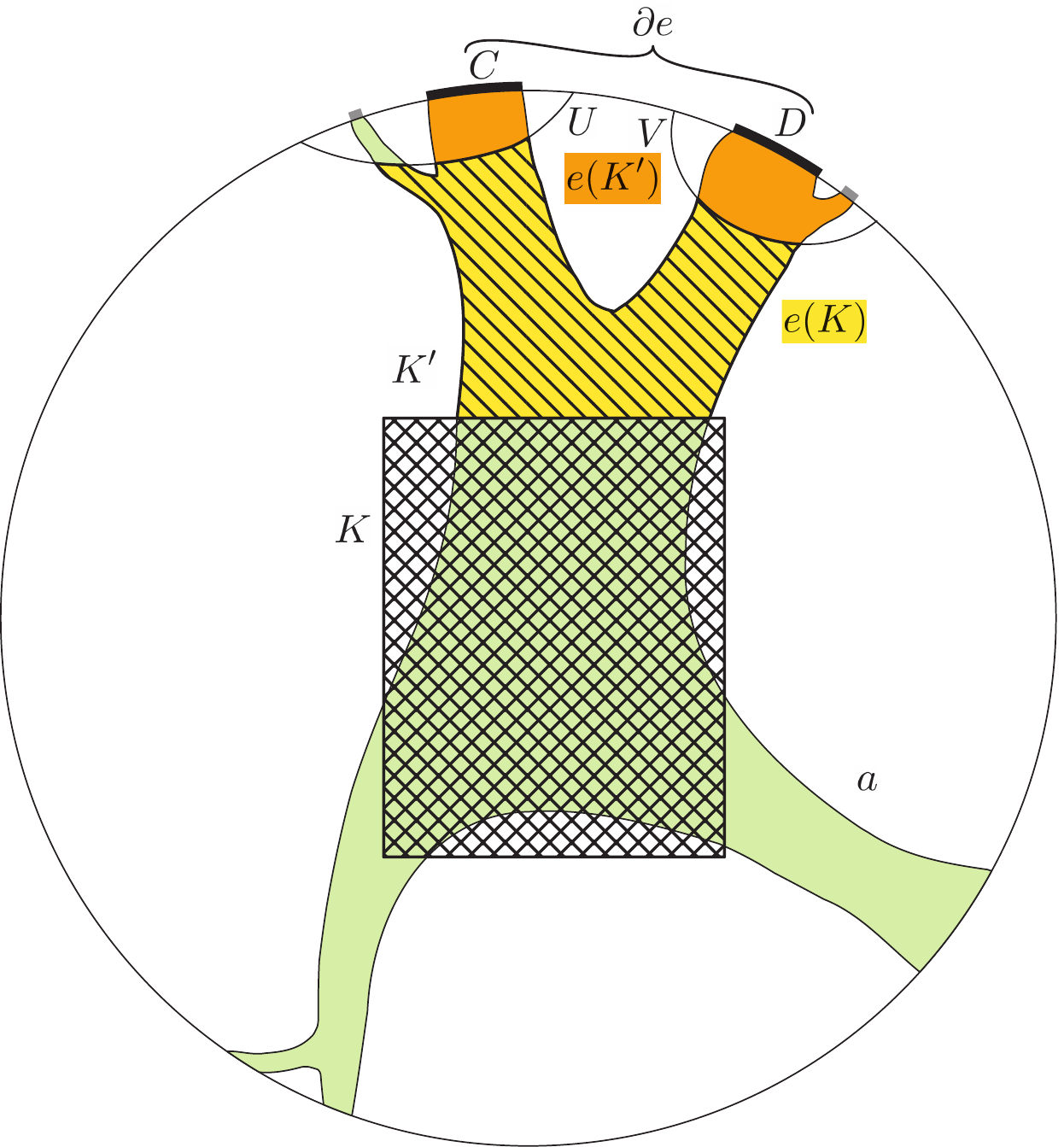}
  \caption{Proof of lemma \ref{sp}.\label{rys:sp}}
\end{figure}
\begin{claim}
There exist $K'$ -- a superset of $K$ such that even $e(K')\subseteq U\cupdot V$.
\end{claim}
\begin{proof}{}
The set $\clcX{e(K)}\sm(U\cup V)$ is a compact subset of $X$, because it is closed in $\cX$ and disjoint with $\bd\cX$.
\par So let $K'=K\cup\left(\clcX{e(K)}\sm(U\cup V)\right)$ be a compact subset of $X$. Then
$$e(K')\subseteq e(K)\sm K'\subseteq e(K)\sm\left(\clcX{e(K)}\sm(U\cup V)\right)\subseteq U\cupdot V,$$
and
$$\bd e(K')\subseteq \bd e(K)\subseteq U\cupdot V,$$
but on the other hand
$$C\cupdot D=\bd e\subseteq \bd e(K').$$
It follows that $\bd e(K')$ intersects both $U$ and $V$. Hence $e(K')\subseteq U\cupdot V$ is not connected, which cotradicts the definition of end.
\end{proof}
That finishes the proof of the lemma.
\end{proof}

\end{document}